\documentclass[10 pt]{amsart}

\usepackage[colorlinks=true,linkcolor=blue,citecolor=blue,urlcolor=blue]{hyperref}
\usepackage{bookmark}
\usepackage{amsthm,thmtools,amssymb,amsmath,amscd}

\usepackage{fancyhdr}
\usepackage{esint}
\usepackage{enumerate}

\usepackage{pictexwd,dcpic}

\swapnumbers
\declaretheorem[name=Theorem,numberwithin=section]{thm}

\declaretheorem[name=Lemma,sibling=thm]{lemma}

\declaretheorem[name=Corollary,sibling=thm]{cor}

\numberwithin{equation}{section}

\usepackage{cleveref}
\crefname{lemma}{Lemma}{Lemmata}
\crefname{prop}{Proposition}{Propositions}
\crefname{thm}{Theorem}{Theorems}
\crefname{cor}{Corollary}{Corollaries}
\crefname{defn}{Definition}{Definitions}
\crefname{example}{Example}{Examples}
\crefname{rem}{Remark}{Remarks}
\crefname{ass}{Assumption}{Assumptions}
\crefname{not}{Notation}{Notation}

\renewcommand{\-}{\bar}

\newcommand{\cn}{\colon}
\newcommand{\sub}{\subset}

\newcommand{\R}{\mathbb{R}}

\renewcommand{\S}{\mathbb{S}}

\newcommand{\Di}{\mathbb{D}}
\newcommand{\B}{\mathbb{B}}
\newcommand{\8}{\infty}

\renewcommand{\a}{\alpha}

\newcommand{\g}{\gamma}
\renewcommand{\d}{\delta}
\newcommand{\e}{\epsilon}

\renewcommand{\l}{\lambda}
\renewcommand{\o}{\omega}

\renewcommand{\O}{\Omega}
\newcommand{\D}{\Delta}

\newcommand{\del}{\partial}

\newcommand{\inpr}[2]{\left\langle #1,#2 \right\rangle}
\newcommand{\fr}[2]{\frac{#1}{#2}}
\newcommand{\x}{\times}

\newcommand{\Theo}[3]{\begin{#1}\label{#2} #3 \end{#1}}
\newcommand{\pf}[1]{\begin{proof} #1 \end{proof}}
\newcommand{\eq}[1]{\begin{equation}\begin{alignedat}{2} #1 \end{alignedat}\end{equation}}

\newcommand{\br}[1]{\left(#1\right)}

\newcommand{\ra}{\rightarrow}

\newcommand{\mc}{\mathcal}
\renewcommand{\it}{\textit}
\newcommand{\mrm}{\mathrm}

\newcommand{\hp}{\hphantom}

\protected\def\ignorethis#1\endignorethis{}
\let\endignorethis\relax

\newcommand{\bb}[1]{\mathbb{#1}}

\newcommand{\pard}[2]{\frac{\partial #1}{\partial #2}}

\newcommand{\ip}[2]{\left \langle #1 , #2 \right\rangle}
\newcommand{\n}{\nabla}

\begin{document}
\title[Free boundary convex hypersurfaces in a ball]{A geometric inequality for convex free boundary hypersurfaces in the unit ball}
\author{Ben Lambert and Julian Scheuer}
\begin{abstract}We use the inverse mean curvature flow with a free boundary perpendicular to the sphere to prove a geometric inequality involving the Willmore energy for convex hypersurfaces of dimension $n\geq 3$ with boundary on the sphere.
\end{abstract}


\address{Ben Lambert, University of Konstanz, Zukunftskolleg, Box 216, 78457 Konstanz, Germany}
\email{benjamin.lambert@uni-konstanz.de}

\address{Julian Scheuer, Albert-Ludwigs-Universit{\"a}t, Mathematisches Institut, Eckerstr.~1,
79104 Freiburg, Germany}
\email{julian.scheuer@math.uni-freiburg.de}
\maketitle
\section{Introduction}
In \cite{LambertScheuer:04/2016} we considered the inverse mean curvature flow perpendicular to the sphere, namely a family of embeddings
\eq{X\cn \Di\x[0,T^*)\ra\R^{n+1},}
where $\Di$ denotes the $n$-dimensional unit disk, which satisfy the Neumann boundary problem
\begin{subequations}\label{IMCF}\begin{align}\dot{X}&=\fr{1}{H}N, \label{Floweq}\\
		X(\del \Di)&=\del X(\Di)\subset\label{Flow2} \S^{n}, \\
		0&=\inpr{N_{|\del\Di}}{\~N(X_{|\del\Di})}, \label{BoundCond}\\
		\inpr{\dot{\g}(0)}{\~N}&\geq 0\quad\forall \g\in C^{1}((-\e,0],M_{t})\cn \g(0)\in\del X(\Di) \label{Flow4}\end{align}\end{subequations}
with initial embedding $X_0$ of a strictly convex hypersurface $M_0,$ also satisfying the conditions \eqref{Flow2}, \eqref{BoundCond} and \eqref{Flow4}. Here $\~N$ denotes the outward unit normal of $\S^n.$ In the following we will refer to these three conditions by saying that $M_0$ is {\it{perpendicular to the sphere from the inside.}}

In \cite[Thm.~1]{LambertScheuer:04/2016} we proved that the maximal time $T^*>0$ until which a smooth solution exists is characterized by the $C^{1,\a}$-convergence of the embeddings $X(t,\cdot)$ to the embedding of a flat disk bisecting the unit ball; also compare \cite[Rem.~1]{LambertScheuer:04/2016}. The aim of this paper is to apply this convergence result to prove a Li-Yau type inequality for convex hypersurfaces with boundary in any dimension $n\geq 3.$ 

\begin{thm}\label{main}
Let $n\geq 3$ and $M^n\subset\bb{R}^{n+1}$ be a smoothly embedded $n$-disk, such that $M^n$ is a convex hypersurface perpendicular to $\S^{n}$ from the inside. Then there holds
\eq{\label{main1}\frac{1}{2}|M|^{\frac{2-n}{n}}\int_{M}H^{2}+\omega_{n}^{\fr{2-n}{n}}|\partial M|\geq \omega_{n}^{\fr{2-n}{n}}|\S^{n-1}|}
and equality holds if and only if $M$ is a perpendicularly intersecting hyperplane.
\end{thm}

Here $|\cdot|$ denotes the respective surface measures of $M,$ $\del M$ and $\S^{n-1}$ as inherited from $\R^{n+1}$ and $\o_n$ is the volume of the $n$-dimensional unit ball. We call a hypersurface $M$ {\it{convex}}, if all the principal curvature at any point are non-negative and {\it{strictly convex},} if they are all positive throughout $M$. Note that convex or strictly convex hypersurfaces with boundary may be way more complicated than in the boundaryless case. In particular the well known supporting hyperplane property in the boundaryless case is not valid without further assumptions if $M$ has nonempty boundary, compare for example the nice treatment of these issues in \cite{Ghomi:/2001}.

In the case of surfaces, $n=2,$ inequalities similar to \eqref{main1} have attracted a lot of attention. In this situation an even sharper version of \eqref{main1} was shown in broader generality than in the restricted class of convex surfaces, and was even demonstrated in higher codimension. Namely, replacing the leading factor $1/2$ in \eqref{main1} by $1/4,$ Volkmann proved the inequality without the convexity assumption in \cite[Prop.~0.2]{Volkmann:/2014}. One of his main results, \cite[Thm.~1.5]{Volkmann:/2014}, can be viewed as a generalization of an inequality by Li and Yau, cf. \cite{LiYau:/1982}, for closed surfaces to the boundary case. In the case of higher dimensions less is known, let us only mention a result by Brendle on minimal surfaces, \cite{Brendle:06/2012}. We refer to the bibliography in \cite{Volkmann:/2014} for a broader overview over the topic. To our knowledge, the inequality \eqref{main1} has not previously been treated in the higher dimensional hypersurface case.     

\section{The case of strictly convex hypersurfaces}
In order to prove \cref{main} we will use the well established strategy to show that the left hand side of \eqref{main1} is decreasing under the flow and then use the convergence result for the flow to show that it limits into the right hand side of \eqref{main1}. For this purpose we need control on the $L^2$-norm of $H.$
\Theo{lemma}{HL2}{
Let the family $(M_t)$ of strictly convex hypersurfaces evolve by \eqref{IMCF}. Then for all $1\leq p<\8$ there holds
\eq{\lim_{t\ra T^{*}}\int_{M_{t}}H^{p}(\cdot,t)=0.}
}

\pf{
From \cite[Lemma~11]{LambertScheuer:04/2016} we know that after a rotation of coordinates
\eq{\ip{N}{e_0}\leq c_0<0}
for some constant $c_0=c_0(M_0).$
Due to the Gaussian formula the height function $w=\ip{X}{e_0}$ satisfies
\eq{\D w=-H\ip{N}{e_0}\geq -c_0H}
and on the boundary it satisfies
\eq{\ip{\n w}{\~N}=w, }
compare \cite[Lemma~7]{LambertScheuer:04/2016}. Hence
\eq{\int_{M_t}H\leq -c_0^{-1}\int_{M_t}\D w=-c_0^{-1}\int_{\del M_t}w\ra 0,\quad t\ra T^*,}
since the boundaries $\del M_t\sub \S^n$ converge to the equator in $C^1.$
The complete result follows due to the boundedness of $H,$ compare \cite[Prop.~2]{LambertScheuer:04/2016}, and interpolation.
}

Now we can prove \cref{main} in the special case of a strictly convex hypersurface, which will also be needed in the proof of the limiting case.

\Theo{lemma}{strictlyconvexmain}{
Let $n\geq 2$ and $M\sub\R^{n+1}$ be a smooth and strictly convex hypersurface perpendicular to $\S^{n}$ from the inside. Then there holds
\eq{\fr{1}{2}|M|^{\fr{2-n}{n}}\int_{M}H^{2}+\o_{n}^{\fr{2-n}{n}}|\del M|> \o_{n}^{\fr{2-n}{n}}|\S^{n-1}|.}
}

\pf{

Under the inverse mean curvature flow perpendicular to the sphere from \cite[Lemma~2.3.1, Lemma~2.3.4]{Gerhardt:/2006} with $M_{0}=M$ as initial hypersurface we have
\eq{\dot{H}=\Delta\br{-\fr{1}{H}}-\fr{\|A\|^{2}}{H}}
and
\eq{\fr{d}{dt}\sqrt{\det(g_{ij})}=\sqrt{\det(g_{ij})}.\label{Volevol}}
Thus 
\eq{\fr{d}{dt}\br{\fr 12\int_{M_{t}}H^{2}d\mu_{t}}&=\int_{M_{t}}H\Delta\br{-\fr 1H}d\mu_{t}-\int_{M_{t}}\|A\|^{2}d\mu_{t}+\fr 12\int_{M_{t}}H^{2}d\mu_{t}\\
		&=-\int_{M_{t}}\fr{\|\nabla H\|^{2}}{H^{2}}d\mu_{t}-\int_{M_{t}}\|A\|^{2}d\mu_{t}+\fr 12\int_{M_{t}}H^{2}d\mu_{t}\\
			&\hp{=}-|\del M_{t}|,}
where we used the divergence theorem and 
\eq{\ip{\n H}{\~N}=-H,}
compare \cite[Lemma~6]{LambertScheuer:04/2016}.
Since
\eq{\|A\|^{2}=\|\mathring{A}\|^{2}+\fr 1n H^{2},}
we have
\eq{\fr 12 H^{2}-\|A\|^{2}=\fr{n-2}{2n}H^{2}-\|\mathring{A}\|^2}
and thus
\eq{\fr{d}{dt}\br{\fr 12\int_{M_{t}}H^{2}d\mu_t}&=-\int_{M_{t}}\fr{\|\nabla H\|^{2}}{H^{2}}d\mu_{t}-\int_{M_{t}}\|\mathring{A}\|^{2}d\mu_{t}\\
		&\hp{=}+\fr{n-2}{2n}\int_{M_{t}}H^{2}d\mu_{t}
			-|\del M_{t}|.}

Furthermore, the volume elements of the induced hypersurfaces
\eq{y_{t}\cn\del \Di\ra\S^{n}} 
satisfy
\eq{\fr{d}{dt}\sqrt{\det(\g_{IJ})}=\fr{\g^{IJ}\eta_{IJ}}{H}\sqrt{\det(\g_{IJ})}< \sqrt{\det(\g_{IJ})},}
where $\g_{IJ}$ and $\eta_{IJ}$ denotes the metric and the second fundamental form of these hypersurfaces respectively. This is due to the fact that they satisfy a related flow equation in the sphere
\eq{\dot{y}=\fr 1H\nu,}
compare \cite[equ.~(20)]{LambertScheuer:04/2016}.
Define
\eq{\label{Q}Q(t)=\fr{1}{2}|M_{t}|^{\fr{2-n}{n}}\int_{M_{t}}H^{2}+\o_{n}^{\fr{2-n}{n}}|\del M_{t}|.}
By the previous calculations we have
\eq{\dot{Q}(t)&<\fr{2-n}{2n}|M_{t}|^{\fr{2-n}{n}}\int_{M_{t}}H^{2}+\fr{n-2}{2n}|M_{t}|^{\fr{2-n}{n}}\int_{M_{t}}H^{2}-|M_{t}|^{\fr{2-n}{n}}|\del M_{t}|\\
			&\hp{=}+\o_{n}^{\fr{2-n}{n}}|\del M_{t}|\\ \label{Qdot}
			&=\br{\o_{n}^{\fr{2-n}{n}}-|M_{t}|^{\fr{2-n}{n}}}|\del M_{t}|\\
			&\leq 0,}
since we already know by \cite[Thm.~1, Rem.~5]{LambertScheuer:04/2016} that $|M_{t}|$ is increasingly converging to $\o_n.$
Furthermore we know by \cref{HL2} that
\eq{\int_{M_{t}}H^{2}\ra 0}
and thus we obtain
\eq{Q(0)> Q(T^{*})=\o_{n}^{\fr{2-n}{n}}|\S^{n-1}|.}
}

We also need the following exact description of the maximal time of existence of a smooth solution to \eqref{IMCF}.

\begin{lemma}[Exact existence time]
Suppose the initial data $M_0$ to \eqref{IMCF} is strictly convex. Then the maximal time of existence $T^*$ is
\label{existtime}
\eq{T^* = \log\left(\frac{\omega_n}{|M_0|}\right).}
In particular we obtain the volume estimate
\eq{|M_0|<\o_n.}
\end{lemma}
\begin{proof}
Using (\ref{Volevol}), we see that $\frac{d}{dt}|M_t| = |M_t|$ and so 
\begin{equation}|M_t| = e^t|M_0|. \label{AreaFormula}\end{equation}
Since we know that the maximal time is when the flow becomes a flat disk and the flow converges in $C^{1, \beta},$ cf. \cite[Rem.~7.4]{LambertScheuer:04/2016}, we know $\omega_n = e^{T^*}|M_0|$ and the equation follows.
\end{proof}

\section{Approximation of weakly convex hypersurfaces}

One of the main difficulties in proving \cref{main} is the lack of information about the IMCF for weakly convex hypersurfaces. The proof of the result in \cite{LambertScheuer:04/2016} makes essential use of the strict convexity. Hence it is not straightforward to obtain the limiting case in \cref{main}. We will use approximation by strictly convex hypersurfaces to overcome this obstacle. To do this we use mean curvature flow with the same Neumann boundary conditions. More specifically, we still assume $M_0$ is parametrised by $X_0:\bb{D} \ra \bb{R}^{n+1}$. Contrary to our previous solution $X$ of the inverse mean curvature flow, we now consider the solution $F:\bb{D}\times[0,T)\ra\bb{R}^{n+1}$ of the mean curvature flow with Neumann boundary condition, i.e. 

\begin{subequations}\label{MCF}\begin{align}\dot{F}&=-HN, \\
		X(\del \Di)&=\del X(\Di)\subset \S^{n}, \\
		0&=\inpr{N_{|\del\Di}}{\~N(X_{|\del\Di})}, \\
		\inpr{\dot{\g}(0)}{\~N}&\geq 0\quad\forall \g\in C^{1}((-\e,0],M_{t})\cn \g(0)\in\del X(\Di) \end{align}
		\end{subequations}
with inital embedding $X_{0}.$

Properties of such mean curvature flows with boundary conditions were studied by A. Stahl in \cite{Stahl:06/1996} and \cite{Stahl:08/1996}. We now use Stahl's short time existence result \cite[Thm. 2.1]{Stahl:06/1996} in conjunction with the following strong maximum principle statement to obtain strictly convex approximating hypersurfaces arbitrarily close to $M_0$ in $C^{2,\alpha}.$ First we need a lemma to ensure that a nontrivial $M$ has a strictly convex point.

\begin{lemma}\label{Height}
Let $M\sub\R^{n+1}$ be a smooth and weakly convex hypersurface perpendicular to $\S^{n}$ from the inside with embedding vector $X$. Then either $\del M$ is an equator of the sphere or there exists $x\in \Di$ such that the second fundamental form of $M$ at $x$ is positive definite.
\end{lemma}

\pf{
By \cite[Lemma~4]{LambertScheuer:04/2016} $\del M\sub\S^n$ is a convex hypersurface of the sphere which is either an equator or strictly contained in an open hemisphere by the classical results in \cite{CarmoWarner:/1970}. In the first case we are done. In the second case we pick a point $e_0\in \mrm{conv}\br{\del M}\sub\S^n,$ where the latter denotes the spherical convex body bounded by $\del M,$ such that also 
\eq{\del M\sub \mrm{int}\br{\mc{H}(e_0)},}
where $\mc{H}(e_0)$ denotes the closed hemisphere with center $e_0$.
By \cite[Lemma~5]{LambertScheuer:04/2016} the height
\eq{w=\ip{X}{e_0}}
over the hyperplane $e_0^{\perp}$ does not attain its global minimum on the boundary of $\Di.$ By attaching a large supporting sphere to $M$ from below we find the existence of a strictly convex point.
}

Now we can prove the approximation result. A similar technique was used in \cite{Hamilton:/1986}.

\begin{thm}\label{Approx}
Suppose $F\cn\bb{D}\times[0,T) \ra \bb{R}^{n+1}$ is a solution to \eqref{MCF} with initial hypersurface $M_0$ being weakly convex and perpendicular to the sphere from the inside. Then either $\del M_0$ is an equator of the sphere or $\br{h_{ij}}>0$ for $t>0$.  
\end{thm}
\begin{proof}
If $\del M_0$ is not an equator, then due to \cref{Height} there exists a strictly convex point.
 
Let 
\eq{\chi(x,t) =\underset{|V|=1} \min h_{ij}V^iV^j.}
Due to the smoothness of $h_{ij}$, $\chi(x,t)$ is Lipschitz continuous in space and therefore by a simple cut-off function argument we find a smooth function $\phi_0:M^n\ra \bb{R}$ so that $0\leq \phi_0 \leq \chi(x,0)$ and there exists $y\in M^n$ so that $\phi_0(y)>0$. We now extend this function to $\phi:\bb{D}^n\times[0,\delta)\ra \bb{R}$ by a heat flow, 
\begin{equation}\label{Heatflow}
 \begin{cases}
  \left(\pard{}{t} - \Delta \right) \phi =0 & \text{on } \text{int}(\bb{D})\times[0,\tau)\\
  \n_\mu \phi = 0 & \text{on } \partial \bb{D}\times [0,\tau)\\
  \phi(\cdot, 0) = \phi_0(\cdot),
 \end{cases}
\end{equation}
where $\D$ is the time dependent Laplace-Beltrami operator of the metrics induced by the solution $F$ of \eqref{MCF}.
This is a linear parabolic PDE and so by standard theory a solution exists for a short time $\tau>0$. By the strong maximum principle (e.g. \cite[Cor.~3.2]{Stahl:06/1996}), for $t>0$ we have $\phi(\cdot,t)>0$ in $\Di$.

We now consider 
\eq{M_{ij} = h_{ij} - \phi g_{ij}}
as long both the MCF and the heat flow exist, say for $0\leq t<\tau$. We know that at time $t=0$ we have $M_{ij}\geq 0$ by construction of $\phi.$ We now aim to apply the weak maximum principle with Neumann boundary conditions, \cite[Thm.~3.3, Lemma~3.4]{Stahl:06/1996}. 

Using the evolution equations in \cite[p.~432]{Stahl:08/1996}, we have that on the flowing manifold 
\begin{align}
 \left(\frac{\partial}{\partial t} - \Delta\right) M_{ij} &= |A|^2h_{ij}-2H h_i^kh_{kj} + 2\phi H h_{ij}=:N_{ij}.
\end{align}
We see that for a unit vector $v$ such that 
\eq{M_{ij}v^i =h_{ij}v^i -\phi g_{ij}v^i=0,} we obtain
\eq{N_{ij}v^iv^j = |A|^2\phi - 2H\phi^2 +2H\phi^2 = |A|^2\phi\geq 0,}
that is, the evolution of $M_{ij}$ satisfies a null eigenvector condition.

For a better comparability to the results in \cite{Stahl:06/1996} and \cite{Stahl:08/1996} we switch to Stahl's notation, so that for $p\in \bb{S}^n$ write $\mu\in T_pM$ for the outward pointing normal to $\bb{S}^n$. Due to \cite[Thm.~4.3 (i)]{Stahl:08/1996}, at a point $p\in\partial M$ for basis tangent vectors $\del_{I}\in T_p M \cap T_p \bb{S}$, the basis
\eq{\mc{B}=\br{\mu,\del_I}_{2\leq I\leq n}}
induces the coordinate representation $M_{I\mu} = 0$. That is $\mu$ is both an eigenvector of $M_{ij}$ and a principal direction at the boundary. We now demonstrate that the conditions of \cite[Lemma 3.4]{Stahl:06/1996} hold. For $\del_I, \del_J \in T_p M \cap T_p\bb{S}^n$, \cite[Thm.~4.3 (ii), (iii)]{Stahl:08/1996} give
\begin{equation}\n_\mu M_{IJ} = h_{\mu\mu}\d_{IJ} - h_{IJ},\qquad \n_\mu M_{\mu\mu} = 2H - nh_{\mu\mu}.
 \label{boundaryMij}
\end{equation}

We suppose first that $V\in T_p \br{\partial M}$ is a minimal eigenvector with eigenvalue $\l\in(-\delta,0]$, that is 
\eq{M_{ij}V^i = \lambda g_{ij}V^i.}
We see that $V$ is also a minimal eigenvector of $h_{ij}$, and therefore 
\eq{h_{ij}V^iV^j\leq h_{\mu\mu}.} Equation (\ref{boundaryMij}) now implies $\nabla_\mu M_{IJ}V^IV^J\geq 0$. 

Now suppose that $\mu$ is a minimal eigenvector with eigenvalue $\l\in(-\delta,0]$. Again minimality of $\mu$ implies that for all $W\in T_p\br{\partial M}$ there holds
\eq{h_{ij}W^iW^j\geq h_{\mu\mu}.}
In particular this implies $H\geq nh_{\mu\mu}$, and so $\n_\mu M_{\mu\mu}\geq H\geq 0$, where we used \cite[Thm.~3.1]{Stahl:08/1996}.

We may now apply \cite[Thm.~3.3, Lemma 3.4]{Stahl:06/1996}, to give that $M_{ij}\geq 0$. Since $\phi>0$ for $t>0$, $h_{ij}>0$ for $\tau>t>0$. This then holds for all time that the flow exists by applying \cite[Prop.~4.5]{Stahl:08/1996} to the mean curvature flow defined by $F(x,t-\frac{\tau}{2})$. 
\end{proof}
\begin{cor} \label{approx}
 Suppose $M$ is a weakly convex hypersurface perpendicular to the sphere from the inside, such that $\del M$ is not an equator. Then there exists an $\epsilon>0$ such that for $0\leq t<\epsilon$ there are smooth and strictly convex hypersurfaces perpendicular to the sphere from the inside and satisfy 
 \[\int_{M_t} H^2 \ra \int_{M} H^2, \ \ |M_t|\ra |M|, \ \ |\partial M_t| \ra |\partial M|\]
 as $t \ra 0$.
\end{cor}
\begin{proof}
 By \cite[Thm.~2.1]{Stahl:06/1996} there exists a solution to equation (\ref{MCF}) for $F\in C^\infty(\bb{D}\times(0,\e))\cap C^{2+\alpha; 1+\frac{\alpha}{2}}(\bb{D}\times[0,\e))$. The convergence then follows due to the regularity of the flow at $t=0$.  
\end{proof}

Now we can prove a crucial estimate for the volume.

\begin{lemma}[Volume estimate]\label{Areaest}
Let $M$ be a weakly convex hypersurface perpendicular to the sphere from the inside such that $\del M$ is not an equator. Let 
\eq{C_M = \{X \in \bb{R}^{n+1} \cn X=\l p, p\in \partial M, 0\leq \l\leq 1\}.}
Then there holds
\eq{|M|\leq |C_M|.} 
\end{lemma}

\pf{
Since $M$ induces a convex hypersurface $\del M$ of the sphere $\S^n,$ as in the proof of \cref{Height} we may pick a point $e_0\in\S^n$ such that
\eq{\del M\sub\mrm{int}(\mc{H}(e_0)),\quad e_0\in\mrm{int}\br{\mrm{conv}_{\S^n}(\del M)},}
where $\mc{H}(e_0)$ denotes the closed hemisphere around $e_0\in\S^n,$
compare \cite[Lemma~4]{LambertScheuer:04/2016}. Approximation of $M$ by strictly convex hypersurface according to \cref{approx} and \cite[Lemma~11]{LambertScheuer:04/2016} yields that $M$ is a standard graph over a convex domain $\-\O\sub \R^n\x \{0\}.$ Due to convexity $M$ lies above each of its tangent hyperplanes and hence the proofs of \cite[Cor.~1-Cor.~3, Lemma~12]{LambertScheuer:04/2016} apply literally to yield 
\eq{\mrm{int}(M)\sub \B^{n+1},}
where the latter denotes the open unit ball of $\R^{n+1}.$ This implies
\eq{M\sub \hat{C}_M:=\{X\in\R^{n+1}\cn X=\l p, p\in \mrm{conv}_{\S^n}(\del M), 0\leq \l\leq 1\}.}
For any $X\in M$ the straight line
\eq{\g(t)=X+tN,\quad t\geq 0,}
must eventually leave $\hat{C}_M,$ since $\ip{N}{e_0}<0.$ Thus $\g$ eventually hits $C_M$ and we see that the least distance projection from the Lipschitz graph $C_M$ to $M$ is surjective. Since this projection is Lipschitz with Lipschitz constant at most $1$, the desired estimate now follows from the area formula, compare \cite[Sec.~3.3]{EvansGariepy:/1992} for example.
}



\section{Proof of \texorpdfstring{\cref{main}}{}}

If $\del M$ is an equator, \eqref{main1} is trivial. Due to \cref{approx} we see that \eqref{main1} now also holds for weakly convex hypersurfaces. So all we have to prove is the characterization of the limit. So suppose that \eqref{main1} holds with equality. If $\del M$ is an equator, then $M$ must be a convex minimal surface, hence totally umbilic and hence a hyperplane. So we may suppose that $\del M$ is not an equator, which in particular implies that 
\eq{|M|\leq |C_M|<\o_n, }
where we used \cref{Areaest}.

Due to \cref{approx} for every $\e>0$ there exists a strictly convex hypersurface perpendicular to the sphere from the inside $M^{\e}$ such that
\eq{Q(M^{\e})\leq Q(M)+\e,}
where $Q(M)$ is the quantity in \eqref{Q} evaluated at the hypersurface $M.$
Starting the flow \eqref{IMCF} with initial hypersurface $M_{\e},$ flow hypersurfaces $M_t^{\e}$ and maximal existence time
\eq{T^{*}_{\e}=\log\br{\fr{\o_n}{|M^{\e}|}},}
in view of \eqref{Qdot} the corresponding quantities $Q^{\e}(t)$ satisfy
\eq{\dot{Q}^{\e}(t)&\leq \br{\o_n^{\fr{2-n}{n}}-|M_t^{\e}|^{\fr{2-n}{n}}}|\del M_t^{\e}|\\
			&=\o_n^{\fr{2-n}{n}}\br{1-e^{\fr{n-2}{n}(T^{*}_{\e}-t)}}|\del M^{\e}_t|.}
Due to \cref{Areaest} and \cref{approx} there exists a positive time $T$ which only depends on $|M|$ and is independent of $\e,$ such that
\eq{T_{\e}^{*}\geq 2T>0.}
Hence for all $\e$ and all $0\leq t\leq T$ there holds
\eq{\dot{Q}^{\e}(t)\leq -c\br{1-e^{\fr{n-2}{n}T}}\equiv -c, }
where $c>0$ only depends on $n,$ $|M|$ and $|\del M|.$
Using \cref{strictlyconvexmain} we obtain, also using that the strict convexity of $M_{\e}$ is preserved, that
\eq{\o_n^{\fr{2-n}{n}}|\S^{n-1}|< Q^{\e}(T)&=Q(M_{\e})+\int_0^T\dot{Q}^{\e}(s)~ds\\
				&\leq Q(M)+\e-cT\\
                &= \o_n^{\fr{2-n}{n}}|\S^{n-1}|+\e-cT, }
giving a contradiction for small $\e$ and completing the proof.

\subsection*{Acknowledgements}
We would like to thank Florian Besau for a hint about the orthographic projection of a spherically convex set.

\bibliographystyle{hamsplain}
\bibliography{Bibliography}
\end{document}